\documentclass[11pt]{article}
\usepackage[letterpaper,twoside,outer=1.3in,vmargin=1.3in,]{geometry}
\usepackage{amsmath,amssymb,amsthm,amsfonts,enumerate,times,epsfig,color,hyperref,cleveref}

\newcommand{\ignore}[1]{}

\newcommand{\tr}{\triangle}

\newcommand{\cuboid}[1]{\mathrm{cuboid}\!\left(#1\right)}
\newcommand{\ind}[1]{\mathrm{ind}\!\left(#1\right)}

\newcommand{\1}{\mathbf{1}}
\newcommand{\0}{\mathbf{0}}

\newtheorem{theorem}{Theorem}

\newcounter{claim_nb}[theorem]
\newtheorem*{claim*}{Claim}

\newcounter{claim_nbs}[section]
\newcounter{subclaim_nb}[claim_nbs]
\title{Testing idealness in the filter oracle model}

\author{Ahmad Abdi \and
G\'erard Cornu\'ejols \and
Bertrand Guenin \and
Levent Tun\c{c}el
}
\begin{document}

\maketitle

\begin{abstract}
A \emph{filter oracle} for a clutter consists of a finite set $V$ along with an oracle which, given any set $X\subseteq V$, decides in unit time whether or not $X$ contains a member of the clutter. Let $\mathfrak{A}_{2n}$ be an algorithm that, given any clutter $\mathcal{C}$ over $2n$ elements via a filter oracle, decides whether or not $\mathcal{C}$ is ideal. We prove that in the worst case, $\mathfrak{A}_{2n}$ must make at least $2^n$ calls to the filter oracle. Our proof uses the theory of cuboids.
\end{abstract}

\paragraph{Background}
Let $V$ be a finite set, and $\mathcal{C}$ a family of subsets of $V$, called \emph{members}. $\mathcal{C}$ is a \emph{clutter} over \emph{ground set} $V$ if no member contains another one~\cite{Edmonds70}. $\mathcal{C}$ is \emph{ideal} if the set covering polyhedron $\left\{x\in \mathbb{R}^V:\sum_{u\in C}x_u\geq 1~\forall C\in \mathcal{C};x\geq \0\right\}$ is integral. The terminology was coined in \cite{Cornuejols94} but the notion goes back to the 1960s by Lehman
\cite{Lehman79} (it took some time for the manuscript to be put in print). 

An important question is the time complexity of detecting the property of idealness. Using basic polyhedral theory, one can show easily that testing idealness belongs to co-NP. In fact, it was shown in~\cite{Ding08} that testing idealness is co-NP-complete, and so testing idealness is NP-hard.

Many examples of clutters from Combinatorial Optimization, such as arborescences, cuts, $T$-joins, and dijoins, have exponentially many members (in the size of the ground set). For this reason, for some problems, it may be more appropriate to work in a model where $\mathcal{C}$ is inputted via an oracle. More precisely, a \emph{filter oracle} for a clutter $\mathcal{C}$ consists of $V$ along with an oracle which, given any set $X\subseteq V$, decides in unit time whether or not $X$ contains a member. 

In the filter oracle model, it is no longer clear that testing idealness belongs to co-NP. Using a seminal theorem of Lehman on \emph{minimally non-ideal} clutters~\cite{Lehman90}, Seymour showed that testing idealness indeed belongs to co-NP~\cite{Seymour90}. In this brief note, we prove that in the filter oracle model, testing idealness cannot be done in polynomial time (regardless of the ``P versus NP" question). 

Our result is proved by using the concept of \emph{cuboids}, initiated in \cite{Abdi-mnp} and developed in ~\cite{Abdi-cuboids}, which allows us to get an understanding of the ``local geometry" of ideal clutters. 

A \emph{cuboid} is a clutter $\mathcal{C}$ whose ground set can be partitioned into pairs $\{u_i,v_i\},i\in [n]$ such that $|\{u_i,v_i\}\cap C|=1$ for all $i\in [n]$ and $C\in \mathcal{C}$. $\mathcal{C}$ can be represented as a subset of $\{0,1\}^n$. More precisely, for each $C\in \mathcal{C}$, let $p(C)$ be the point in $\{0,1\}^n$ such that $p(C)_i=0$ iff $C\cap \{u_i,v_i\}=\{u_i\}$. Let $S:=\{p(C):C\in \mathcal{C}\}$. 
We call $\mathcal{C}$ the \emph{cuboid of $S$}, denote by $\cuboid{S}:=\mathcal{C}$ and by $C(p)$ the member of $\mathcal{C}$ corresponding to $p\in \{0,1\}^n$. Note that the operator $\cuboid{\cdot}$ takes any subset of $\{0,1\}^n$ to a cuboid. $S$ is \emph{cube-ideal} if $\cuboid{S}$ is an ideal clutter. It is known that $S$ is cube-ideal iff the convex hull of $S$ can be described by $\0\leq x\leq \1$ and inequalities of the form $\sum_{i\in I}x_i+\sum_{j\in J}(1-x_j)\geq 1$ for disjoint $I,J\subseteq [n]$~\cite{Abdi-mnp,Abdi-cuboids}. Thus, the set $\{0,1\}^n$ is cube-ideal. Moreover, if $S$ is cube-ideal, then so is every restriction of it obtained by fixing coordinates to $0$ or $1$ (and then dropping the coordinates).

Let $p\in \{0,1\}^n$. The set $S\tr p$ is defined as $\{x\tr p:x\in S\}$, where the second $\tr$ denotes coordinate-wise sum mod $2$; we call $S\tr p$ the \emph{twisting of $S$ with respect to $p$}. It can be readily seen that twisting preserves cube-idealness. The \emph{induced clutter of $S$ with respect to $p$}, denoted by $\ind{S\tr p}$, is the clutter over ground set $[n]$ whose members are the inclusionwise minimal sets in $\{C\subseteq [n]:\chi_C\in S\tr p\}$. In particular, if $p\in S$, then $\ind{S\tr p}=\{\emptyset\}$. A key insight for this note is that $S$ is cube-ideal iff the induced clutter of $S$ with respect to every point in $\{0,1\}^n$ is ideal~\cite{Abdi-cuboids}. Consequently, if for example $S$ excludes a single point $p$ of $\{0,1\}^n$, then $S$ is cube-ideal, because $\ind{S\tr p}=\{\{1\},\{2\},\ldots,\{n\}\}$ is an ideal clutter.

\paragraph{The result} We are almost ready to prove the main result of this note. Let $n\geq 1$ be an integer, and let $G_n$ denote the skeleton graph of the unit hypercube $[0,1]^n$. Given $S\subseteq \{0,1\}^n$, if $G_n[\{0,1\}^n-S]$ has maximum degree at most $2$, then $S$ is cube-ideal. This result was first proved in~\cite{Cornuejols16}, and further studied in \cite{Abdi-resistant}. It can also be readily shown using the characterization of cube-idealness in terms of induced clutters. The result, however, does not extend from $2$ to $3$. Let $S_3:=\{e_1+e_2, e_2+e_3,e_1+e_3,e_1+e_2+e_3\}\subseteq \{0,1\}^3$. Then $S_3$ is not cube-ideal because its convex hull has a facet-defining inequality of the form $x_1+x_2+x_3\geq 2$. Moreover, in $G_3[\{0,1\}^3-S_3]$, the vertex $\0$ has $3$ neighbours $e_1,e_2,e_3$. 

\begin{theorem}\label{filter-oracle-exp}
Let $\mathfrak{A}_{2n}$ be an algorithm that, given any clutter $\mathcal{C}$ over $2n$ elements via a filter oracle, decides whether or not $\mathcal{C}$ is ideal. Then in the worst case, $\mathfrak{A}_{2n}$ must make at least $2^n$ calls to the filter oracle.
\end{theorem}
\begin{proof}
For all $p\in \{0,1\}^n$ and distinct $i,j,k\in [n]$, let $S_{(p:i,j,k)} := \{0,1\}^n - \{p,p\tr e_i,p\tr e_j,p\tr e_k\}.$ Then $S_{(p:i,j,k)}$ is not cube-ideal as it has an $S_3$ restriction, while every proper superset $S'$ of $S_{(p:i,j,k)}$ is cube-ideal as $G_n[\{0,1\}^n-S']$ has degree at most $2$. In particular, $\cuboid{S_{(p:i,j,k)}}$ is a non-ideal clutter, while $\cuboid{S'}$ is ideal for every $S'\supsetneq S$. Thus, $\mathfrak{A}_{2n}$ must distinguish between $\cuboid{S_{(p:i,j,k)}}$ and $\cuboid{S'}$ for every $S'\supsetneq S$. Consequently, for every point $q\in \{p,p\tr e_i,p\tr e_j,p\tr e_k\}$, the algorithm must query the set $C(q)$ or a superset of it. In fact, for $q\in \{p,p\tr e_i,p\tr e_j,p\tr e_k\}-\{p\}$, every neighbour of $q$ in $G_n$ except for $p$ belongs to both $S_{(p:i,j,k)}$ and $S',S'\supsetneq S$, so the algorithm must query either $C(q)$ or $C(q)\cup C(p)$ (note that $|C(q)\cup C(p)|=|C(q)|+1$).

By applying the argument above to every $p\in \{0,1\}^n$ and distinct $i,j,k\in [n]$, we conclude the following: For every $q\in \{0,1\}^n$ and every neighbour of it $p\in \{0,1\}^n$ in $G_n$, $\mathfrak{A}_{2n}$ must query at least one of $C(q),C(q)\cup C(p)$. It can be readily checked that $\mathfrak{A}_{2n}$ must query at least $2^{n}$ sets.
\end{proof}

Let $\mathcal{C}$ be a clutter over ground set $V$. Let $I,J$ be disjoint subsets of $V$. The \emph{minor of $\mathcal{C}$ obtained after deleting $I$ and contracting $J$}, denoted $\mathcal{C}\setminus I/J$, is the clutter over ground set $V-(I\cup J)$ whose members are the inclusionwise minimal sets in $\{C-J:C\in \mathcal{C},C\cap I=\emptyset\}$. Given a filter oracle for $\mathcal{C}$, we also have one for every minor $\mathcal{C}\setminus I/J$~\cite{Seymour90}.

Being ideal is closed under taking minor operations~\cite{Seymour77}. Two clutters are \emph{isomorphic} if one can be obtained from the other by relabeling its ground set. Denote by $\Delta_3$ any clutter isomorphic to $\{\{1,2\},\{2,3\},\{3,1\}\}$. It can be readily checked that $\Delta_3$ is the only non-ideal clutter over a ground set of size at most three. In particular, if a clutter has a $\Delta_3$ minor, then it is non-ideal. 

Let $S\subseteq \{0,1\}^n$. It can be readily seen that every induced clutter of $S$ is a (contraction) minor of $\cuboid{S}$. Thus, since $\ind{S_3}=\{\{1,2\},\{2,3\},\{1,3\}\}$, $\cuboid{S_3}$ has a $\Delta_3$ minor, proving once again that $S_3$ is not cube-ideal. It can also be readily seen that if $R$ is a restriction of $S$, then $\cuboid{R}$ is a minor of $\cuboid{S}$. Consequently, in the proof of \Cref{filter-oracle-exp}, it can be readily seen that $\cuboid{S_{(p:i,j,k)}}$ has a $\Delta_3$ minor, while $\cuboid{S'}$ is ideal and therefore has no $\Delta_3$ minor for every $S'\supsetneq S$. Thus, the proof also implies the following.

\begin{theorem}
Let $\mathfrak{D}_{2n}$ be an algorithm that, given any clutter $\mathcal{C}$ over $2n$ elements via a filter oracle, decides whether or not $\mathcal{C}$ has a $\Delta_3$ minor. Then in the worst case, $\mathfrak{D}_{2n}$ must make at least $2^n$ calls to the filter oracle.\qed
\end{theorem}

%
%

{\small \bibliographystyle{abbrv}\bibliography{references}}

\end{document}